                    \def\version{December 12, 2022}                       %

 \documentclass[reqno,11pt]{amsart}
 \usepackage{amsmath, amsthm, a4, latexsym, amssymb}
\usepackage[unicode]{hyperref}
\usepackage{srcltx}

\usepackage{mathtools}
\usepackage{fancyhdr}
\usepackage{url}
\usepackage{float}
\usepackage{color}
\usepackage{subfigure}
\usepackage[titletoc]{appendix}
\usepackage{hyperref}

\usepackage{nicefrac}
\usepackage{mathrsfs}
\usepackage{bbm}

\def\@rmrk#1#2{\refstepcounter
    {#1}\@ifnextchar[{\@yrmrk{#1}{#2}}{\@xrmrk{#1}{#2}}}

\makeatletter\@addtoreset{equation}{section}\makeatother

 \newfont{\bfit}{cmbxti10 scaled 1200}

\renewcommand{\d}{{\rm d}}

 \newcommand{\e}{{\rm e} }

 \newcommand{\eps}{\varepsilon}

 \newcommand{\R}{\mathbb{R}}
 \newcommand{\N}{\mathbb{N}}
 \newcommand{\Z}{\mathbb{Z}}

 \newcommand{\E}{\mathbb{E}}
 \renewcommand{\P}{\mathbb{P}}
 \def\1{{\mathchoice {1\mskip-4mu\mathrm l} 
{1\mskip-4mu\mathrm l}
{1\mskip-4.5mu\mathrm l} {1\mskip-5mu\mathrm l}}}

\renewcommand{\subsection}{\secdef \subsct\sbsect}
\newcommand{\subsct}[2][default]{\refstepcounter{subsection}
\vspace{0.15cm}
{\flushleft\bf \arabic{section}.\arabic{subsection}~\bf #1  }
\nopagebreak\nopagebreak}
\newcommand{\sbsect}[1]{\vspace{0.1cm}\noindent
{\bf #1}\vspace{0.1cm}}

{\nopagebreak {\hfill\rule{2mm}{2mm}}\\ }

\newtheorem{theorem}{Theorem}[section]
\newtheorem{lemma}[theorem]{Lemma}

\newtheoremstyle{thm}{1.5ex}{1.5ex}{\itshape\rmfamily}{}
{\bfseries\rmfamily}{}{2ex}{}

\newtheoremstyle{rem}{1.3ex}{1.3ex}{\rmfamily}{}
{\itshape\rmfamily}{}{1.5ex}{}
\theoremstyle{rem}

\refstepcounter{subsubsection}

\def\thebibliography#1{\section*{References}
  \list%
  {\arabic{enumi}.}
    {\settowidth\labelwidth{[#1]}\leftmargin\labelwidth
    \advance\leftmargin\labelsep
    \parsep0pt\itemsep0pt
    \usecounter{enumi}}
    \def\newblock{\hskip .11em plus .33em minus .07em}
    \sloppy                   
    \sfcode`\.=1000\relax}

\begin{document}
\title[Positive and negative moments for directed polymers in random environment in weak disorder]
{\large Positive and negative moments for directed polymers in random environment in weak disorder}
\author[Rodrigo Bazaes and Chiranjib Mukherjee]{}
\maketitle
\thispagestyle{empty}
\vspace{-0.5cm}

\centerline{\sc Rodrigo Bazaes\footnote{Universit\"at M\"unster, Einsteinstrasse 62, M\"unster 48149, Germany {\tt rbazaes@uni-muenster.de}} and 
Chiranjib Mukherjee\footnote{Universit\"at M\"unster, Einsteinstrasse 62, M\"unster 48149, Germany {\tt chiranjib.mukherjee@uni-muenster.de}}}

\renewcommand{\thefootnote}{}

\vspace{-0.5cm}
\centerline{\textit{Universit\"at M\"unster}}
\vspace{0.2cm}

\begin{center}
\version
\end{center}

\begin{quote}{\small {\bf Abstract: }
Very recently, Junk \cite{J22} showed that for directed polymers in bounded random environments, the weak disorder (uniform integrable) phase implies that 
the polymer martingale is bounded in $L^p$ for some $p>1$ and also in $L^q$ for some $q<0$. Here, we establish this characterization of the weak disorder phase without requiring 
the boundedness assumption on the environments. 
}
\end{quote}

\section{Introduction}
We consider the model of {\it directed polymers in random environment} described by the following setup. Fix a collection $\{\omega(t,x):t\in \N,x\in \Z^d\}$ of independent and identically distributed (i.i.d.)
random variables on some probability space $(\Omega,\mathcal{F},\mathbb{P})$ satisfying
\begin{equation}\label{eq-exp-mom-assump}
	 \E[\e^{\beta|\omega(1,0)}|]<\infty \qquad\qquad \forall \beta\geq 0.
\end{equation}
In the above display and in the sequel, $\E=\E^\P$ will stand for the expectation w.r.t. $\P$. 
We are interested in the {\it renormalized partition function} defined by 
 \begin{equation}\label{eq:mart-def}
W_n= W_n(\beta):=E_0\left[\e^{\beta\sum_{i=1}^n \omega(i,S_i)-n\lambda(\beta)}\right],
\end{equation}
where $\lambda(\beta):=\log \E[\e^{\beta \omega(1,0)}]$ and $(S_n)_{n\geq 0}$ is a simple random walk on $\Z^d$ starting at the origin, whose law is denoted by $P_0$ and corresponding expectation denoted by $E_0$. 
As already observed by Bolthausen \cite{B89}, for any $\beta\geq 0$, 
$(W_n)_{n\geq0}$ is a (non-negative) martingale w.r.t. the filtration generated by the i.i.d. random variables $\{\omega(t,x):t\in \N,x\in \Z^d\}$. 
Therefore, for any $\beta\geq 0$, $W_n$ converges almost surely to a non-negative random variable $W_\infty=W_\infty(\beta)$. It is not difficult to see that the 
event $\{W_\infty=0\}$ belongs to the tail $\sigma$-field generated by the i.i.d. random variables $\{\omega(t,x):t\in \N,x\in \Z^d\}$ and therefore 
$\P[W_\infty>0] \in \{0,1\}$. It is well-known that (\cite{IS88,B89,CY06,CSY04,MSZ16})
when $d\geq 3$, there is $\beta_c=\beta_c(d)\in (0,\infty)$ (see \eqref{eq-crit-beta} for a precise definition) so that for 
 for $\beta \in (0,\beta_c)$, $W_\infty$ is non-degenerate and 
\begin{equation}\label{eq-limit-mart}
	W_\infty=\lim_{n\to\infty}W_n(\beta)>0\qquad \mathbb{P}\text{-a.s.,}
\end{equation}
while for $\beta>\beta_c$, $W_\infty=0$. The phases $\beta \in (0,\beta_c)$ (resp. $\beta>\beta_c$) are known as the {\it weak} (resp. {\it strong}) {\it disorder}. 
Furthermore, weak disorder is equivalent to uniform
integrability of the martingale $(W_n)_{n\in \N}$ \cite[Proposition 3.1]{CY06}.

Recently, a new characterization of the weak disorder phase was established in an important work of Junk \cite{J22}, who showed that if {\it the random environment is bounded} -- that is, if we assume that 
\begin{equation}\label{eq:bounded-above-env}
	\text{there exists some } K>0\text{ such that }\mathbb{P}(\omega(1,0)\leq K)=1,
\end{equation}
then uniform integrability of the martingale $(W_n)_n$ implies its $L^p(\P)$-boundedness for some $p>1$ (i.e., $\sup_{n\geq 0} \E[W_n^p]< \infty$). 
It was also shown there that under the assumption  
\begin{equation}\label{eq:bounded-below-env}
	\text{there exists some } K>0\text{ such that }\mathbb{P}(\omega(1,0)\geq -K)=1, 
\end{equation}
the uniform integrability of $(W_n)_n$ implies that it is $L^{q}(\P)$-bounded for some $q<0$. The purpose of our article is to show that these properties 
continue to remain true even if we drop the boundedness assumption. Here is our main result: 
\begin{theorem}\label{thm-main}
	Fix $d\geq 3$ and $\beta>0$ so that the martingale $(W_n)_n$ satisfies \eqref{eq-limit-mart}. Then there exist $p_0>1$ and $q_0<0$ such that 
	\begin{itemize}
		\item for all $1<p<p_0$, the martingale $(W_n)_{n\geq 0}$ is $L^p(\P)$-bounded,
		\item for all $q_0<q<0$, the  martingale $(W_n)_{n\geq 0}$ is $L^q(\P)$-bounded. 
	\end{itemize}
\end{theorem}
To put our work into context, let us mention that the uniform integrability criterion for weak disorder does not provide any 
{\it closed-form} characterization of the critical temperature $\beta_c$ and, in practice, the uniform integrability is also not easy to analyze. Several important features of the weak disorder phase have been established 
(e.g. \cite{CY06,BC20}) but our understanding of the directed polymers is much more complete and a lot more works are available for a different, {\it very high temperature phase}, which is characterized by the $L^2$-boundedness of the martingale $(W_n)_n$, which is strictly stronger than uniform integrability, see \eqref{eq:L2-beta} below. While the very high temperature regime is computationally quite convenient, the result of \cite{J22} shows that, for bounded environments, no {\it true} phase transition occurs within the weak disorder phase at the $L^2$ threshold. Theorem \ref{thm-main} above reconfirms this intuition and underlines that the characterization of weak disorder via moments is an {\it intrinsic} property of directed polymers in a random environment, independent of any boundedness of the underlying environment -- in particular, Theorem \ref{thm-main} establishes this characterization of weak disorder for natural choices of unbounded environments like Gaussian, Poisson etc.

Let us briefly comment on the main idea of the proof. An important ingredient for the argument in \cite{J22} is a result concerning the running maximum of uniformly integrable martingales (see Theorem 2.1 there). Such arguments 
have been quite useful in the context of branching processes \cite{AN} and branching random walks \cite{Biggins} where the property 
\begin{equation}\label{K}
\frac{M_{n+1}}{M_n} \leq K \qquad \mbox{a.s. for some} \,\, K>0,
\end{equation} 
(here $(M_n)_n$ is the associated martingale in the respective setup) 
implies that the martingale converges in $L^p$ for some explicit $p>1$. If one assumes that for any convex function $f$, $\E[ f( \frac{M_{n+m}}{M_n}) | \mathcal F_n] \leq \E[f(M_m)]$, then 
the branching structure provides enough independence between particles in generation $n$, implying the $L^p$ boundedness since the influence from the early generations is small. 
This idea was used in \cite{J22} for directed polymers in a random environment, where in contrast to the branching structure, the partition function carries long-range correlations. Combined with the boundedness assumption 
(note that \eqref{eq:bounded-above-env} implies \eqref{K}) the $L^p$ boundedness was then established there. 
While the latter assumption is not available in the present setup, we will combine the previous approach with some of the ideas from our recent work \cite{BLM22}, where a similar statement as Theorem \ref{thm-main} 
was obtained for Gaussian multiplicative chaos in the Wiener space (also constructed there). 
However, in the continuous setup, contributions coming from the ``jumps" and the stopping times that we use get flushed out, which contrasts the present discrete setup. 
Finally, we remark that, under the aforementioned boundedness assumption, it has been shown also recently in \cite[Corollary 1.3]{J22b} that for weak disorder, $(W_n)_n$ remains bounded in $L^p(\P)$ for 
$p\in (1, 1+ \frac 2 d)$. This result heavily depends on the fluctuation results \cite[Theorem 1.1]{J22} obtained there for which the boundedness assumption seems to play perhaps an even more important role. It is conceivable 
that combined with some of the present ideas, a similar statement could also be shown without requiring boundedness.

The rest of the article is devoted to the proof of Theorem \ref{thm-main}. 

\section{Proof of Theorem \ref{thm-main}}
For the proof of the theorem, we will use the following results from \cite{J22}:
\begin{lemma}\label{lemma:convex-up-bound}
Let $f:\R_+\mapsto \R$ be a convex function, and $m,n\in \N$. If $\mathcal F_n$ denotes the $\sigma$-algebra generated by the i.i.d. random variables 
	$\big(\omega(i,x): i \leq n, x \in \Z^d\big)$, then 
	\begin{equation}\label{eq:a.s.bound-convex}
		\E\left[f\left(\frac{W_{n+m}}{W_n}\right)\bigg \vert\mathcal{F}_n \right]\leq \E[f(W_m)]\qquad a.s.
	\end{equation}
	In particular,\begin{equation}
		\E\left[f\left(\frac{W_{n+m}}{W_n}\right)\right]\leq \E[f(W_m)]\leq \E[f(W_{n+m})].\label{eq:moment-bound-convex}
	\end{equation}
	The inequality is reversed for a concave function. 
\end{lemma}

\begin{lemma}\label{lemma:running-max}	
Given $n\in \N$, let $M_n:=\max_{0\leq j\leq n}W_n$ and $M_\infty:=\sup_{n\geq 0}W_n$. If the martingale $W_n$ is uniformly integrable, then $\E[M_\infty]<\infty$.
\end{lemma}
The proofs can be found in Appendix \ref{sec-appendix}.

\subsection{Proof of Theorem \ref{thm-main}: positive moments.}
Let \begin{equation}\label{eq-crit-beta}
	\beta_c:=\sup\Big\{\beta>0: (W_n)_{n\geq 0}\text{ is uniformly integrable}\Big\}
\end{equation}
and \begin{equation}\label{eq:L2-beta}
	\beta_{L^2}:=\sup\Big\{\beta>0: \sup_n \E[W_n^2]< \infty\Big\}.
\end{equation}
It is well-known that $0<\beta_{L^2}<\beta_c$ when $d\geq 3$ \cite{Birkner,C17}. Moreover, if $\beta<\beta_{L^2}$, then we can take $p_0=2$.$^1$\footnote{$^1$Note that Theorem \ref{thm-main} also implies that for the  directed polymer model, for $\beta \in (0,\beta_{L^2})$, there is $\eps>0$ such that $\sup_n \E[W_n^{2+\eps}]<\infty.$}  Thus, we can assume that $\beta_{L_2}\leq \beta<\beta_c$, and therefore we need to find $p_0\in (1,2)$ such that $W_{\infty}\in L^p$ for all $1<p<p_0$. Given $t>1$, define the stopping time \begin{equation}
	\begin{aligned}
		\tau_t&:=\inf\Big\{n\geq 0: W_n>t\Big\}.
	\end{aligned}
\end{equation}
For a given $n\in \N$ and $t>1$, $p\in (1,2)$ to determine, we have \begin{equation}\label{eq:eq1}
	\begin{aligned}
		\E[W_n^p]&=\E[W_n^p,\tau_t>n]+\E[W_n^p,\tau_t \leq n]\\&\leq t^p+\E[W_n^p,\tau_t \leq n]\\
		&= t^p+\sum_{i=1}^n\E[W_n^p,\tau_t=i]\\
		&=t^p+\sum_{i=1}^n\E\left[\left(\frac{W_n}{W_i}\right)^p\left(\frac{W_i}{W_{i-1}}\right)^p W_{i-1}^p,\tau_t=i\right]\\
		&\leq t^p+t^p\sum_{i=1}^n\E\left[\left(\frac{W_n}{W_i}\right)^{p}\left(\frac{W_i}{W_{i-1}}\right)^p\tau_t=i\right].
	\end{aligned}
\end{equation} 
In the last line, we used that on the event $\{\tau_t=i\}$, $W_{i-1}^p\leq t^p$. Conditioning on $\mathcal{F}_i$ and using Lemma \ref{lemma:convex-up-bound}, we obtain \begin{equation}\label{eq:eq2}
	\begin{aligned}
		\E[W_n^p]&\leq  t^p+t^p\sum_{i=1}^n\E\left[\left(\frac{W_i}{W_{i-1}}\right)^p\mathbbm{1}\{\tau_t=i\}\E\left[\left(\frac{W_n}{W_i}\right)^{p}\bigg\vert \mathcal{F}_i\right]\right]\\
		&\leq t^p+t^p\E[W_n^p]\sum_{i=1}^n\E\left[\left(\frac{W_i}{W_{i-1}}\right)^p\mathbbm{1}\{\tau_t=i\}\right]
	\end{aligned}
\end{equation}

Moreover, if $q,r>1$ (to be chosen later) satisfy $\frac{1}{q}+\frac{1}{r}=1$, by H\"{o}lder's inequality, \begin{equation}\label{eq:eq6}
	\begin{aligned}
		\sum_{i=1}^n\E\left[\tau_t=i, \left(\frac{W_{i}}{W_{i-1}}\right)^{p}\right]&\leq \sum_{i=1}^n\E\left[ \left(\frac{W_{i}}{W_{i-1}}\right)^{pq}\right]^{1/q}\P(\tau_t=i)^{1-\frac{1}{q}}.\\
	\end{aligned}
\end{equation}
By Lemma \ref{lemma:convex-up-bound}, $\E\left[ \left(\frac{W_{i}}{W_{i-1}}\right)^{pq}\right]^{1/q}\leq \E\left[W_1^{pq}\right]^{1/q}$. Since $p<2$,\begin{equation*}
	\E\left[W_1^{pq}\right]^{1/q}\leq \E[W_1^{2q}]^{\frac{p}{2q}}\leq \E[W_1^{2q}]
\end{equation*}
(recall that $\E[W_1^{2q}]\in [1,\infty)$ by \eqref{eq-exp-mom-assump}). The last display and \eqref{eq:eq6} yield \begin{equation}\label{eq:varphi-def}
	\sum_{i=1}^n\E\left[\tau_t=i, \left(\frac{W_{i}}{W_{i-1}}\right)^{p}\right]\leq \E[W_1^{2q}]\sum_{i=1}^n\P(\tau_t=i)^{1-\frac{1}{q}},
\end{equation}
and therefore (recall \eqref{eq:eq2}),\begin{equation}\label{eq:eq7}
	\E[W_n^p]\leq  t^p+t^p\E\left[W_{n}^p\right]\E[W_1^{2q}]\sum_{i=1}^n\P(\tau_t=i)^{1-\frac{1}{q}}.
	\end{equation}

Let \begin{equation*}\label{eq:varphi-def}
	 \varphi(t,r):=\sum_{i=1}^\infty\P(\tau_t=i)^r\in[0,\infty], \qquad t>1, r\in (0,1].
\end{equation*}

Note that for all $t>1$, $\varphi(t,1)=\sum_{i=1}^\infty\P(\tau=i)= \P(\tau_t<\infty)=\P(M_\infty>t)$ (recall Lemma \ref{lemma:running-max}), and \begin{equation*}
	\E[M_\infty]=\int_{0}^{\infty}\P(M_\infty>t)\d t=1+\int_{1}^\infty\P(M_\infty>t)<\infty.
\end{equation*}
We conclude that \begin{equation*}
	\int_{1}^\infty \varphi(t,1)\d t= c:=\int_1^\infty \P(M_\infty>t)\d t<\infty.
\end{equation*}
By monotone convergence, $\lim_{r\nearrow 1}\int_{1}^\infty \varphi(t,r)\d t=c$. In particular, for $r>0$ close to $1,\int_{1}^\infty \varphi(t,r)\d t<\infty$. As a consequence, if $q>1$ is sufficiently large, then for all $\eps>0$ there exists some $t_0=t_0(q,\eps)>1$ satisfying \begin{equation*}
	\sup_n\sum_{i=1}^n\P(\tau_{t_0}=i)^{1-\frac{1}{q}}\leq \varphi\left(t_0,1-\frac{1}{q}\right)<\frac{\eps}{t_0}.\end{equation*}
For $q$ as above, set 
$$
\eps=\frac{1}{2 \E[W_1^{2q}]},
$$
 so that for some $t_0>1$,
 \begin{equation*}
 \begin{aligned}
 t_0^p \, \E\left[(W_{n})^p\right] \, \E[W_1^{2q}]\,\bigg(\sum_{i=1}^n\P(\tau_{t_0}=i)^{1-\frac{1}{q}}\bigg)
 <\frac{t_0^{p-1}}{2}.
 \end{aligned}
\end{equation*}
Finally, choose $p\in (1,2)$ such that $t_0^{p-1}<2$. The last display and \eqref{eq:eq7} imply 
$$
\E[W_n^p]\leq t_0^p+\frac{t_0^{p-1}}{2}\E[W_n^p],
$$
 and hence \begin{equation*}
	\sup_n\E[W_n^p]\leq \frac{2t_0^p}{2-t_0^{p-1}}<\infty.
\end{equation*}
\qed

\subsection{Proof of Theorem \ref{thm-main}: negative moments}
We will find $q\in (0,1)$ such that $\sup_n\E[W_n^{-q}]<\infty$. The proof is similar to the previous case. For $t>1$ to determine, set 
\begin{align*}
\tau_t &:=\inf\{n\geq 0: W_n\leq 1/t\}.
\end{align*}

 We obtain a decomposition as in \eqref{eq:eq1}:\begin{equation}\label{eq:eq8}
	\begin{aligned}
		\E[W_n^{-q}]&=\E[W_n^{-q},\tau_t>n]+\E[W_n^{-q},\tau_t \leq n]\\&\leq t^q+\E[W_n^{-q},\tau_t \leq n]\\
		&=t^q+\sum_{i=1}^n\E[W_n^{-q},\tau_t=i]\\
		&=t^q+\sum_{i=1}^n\E\left[\left(\frac{W_n}{W_i}\right)^{-q}\left(\frac{W_i}{W_{i-1}}\right)^{-q}W_{i-1}^{-q},\tau_t=i\right]\\
		&\leq t^q+t^q\sum_{i=1}^n\E\left[\left(\frac{W_n}{W_i}\right)^{-q}\left(\frac{W_i}{W_{i-1}}\right)^{-q},\tau_t=i\right],
		\end{aligned}
\end{equation} 
where in the last line we used that $W_{i-1}\geq \frac{1}{t}$ on the event $\{\tau_t=i\}$. Conditioning on $\mathcal{F}_i$ and using Lemma \ref{lemma:convex-up-bound}, we deduce that 
\begin{equation}\label{eq:eq9}
\E[W_n^{-q}]\leq t^q+t^q\E[W_n^{-q}]\sum_{i=1}^n\E\left[\left(\frac{W_i}{W_{i-1}}\right)^{-q},\tau_t=i\right].	
\end{equation}

Next, we apply H\"{o}lder's inequality with exponents $\frac{1}{q}$ and $\frac{1}{1-q}$, and \eqref{eq:moment-bound-convex} with the convex function $x\mapsto x^{-1}$ to obtain \begin{equation}\label{eq:eq10}
	\begin{aligned}
		\E[W_n^{-q}]&\leq t^q+t^q\E[W_n^{-q}]\sum_{i=1}^n\E\left[\left(\frac{W_i}{W_{i-1}}\right)^{-1}\right]^q\P(\tau_t=i)^{1-q}\\
	&\leq t^q+t^q\E[W_n^{-q}]\E[W_1^{-1}]\sum_{i=1}^\infty\P(\tau_t=i)^{1-q}.
	\end{aligned}
\end{equation}
Note that 
$$
\varphi(t):=\sum_{i=1}^\infty\P(\tau_t=i)=\P(\tau_t<\infty)\to 0\qquad \mbox{as }\,\,t\to\infty.
$$
 Therefore, for $t_0>1$ large enough, 
 $$
 \varphi(t_0)<\frac{1}{4\E[W_1^{-1}]}.
 $$
  Since $\varphi(t_0)=\lim_{r\nearrow 1}\sum_{i=1}^\infty\P(\tau_{t_0}=i)^{r}$, we can find some $q_0\in (0,1)$ such that \begin{equation*}
	t_0^{q_0}<2\qquad \text{ and } \qquad \sum_{i=1}^\infty\P(\tau_{t_0}=i)^{1-q_0}< \frac{1}{2\E[W_1^{-1}]}.
\end{equation*}
The last display, combined with \eqref{eq:eq10} implies that for all $n$ and $q_0\in (0,1)$ as above,\begin{equation*}
	\E[W_n^{-q_0}]\leq \frac{4}{2-t_0^{q_0}}<\infty.
\end{equation*}

This concludes the proof of Theorem \ref{thm-main}.\qed

\appendix 

\section{Proof of Lemmas \ref{lemma:convex-up-bound} and \ref{lemma:running-max}.}\label{sec-appendix}
While the proofs can be found in \cite{J22}, we include them here for the sake of completeness. 

\begin{proof}[Proof of Lemma \ref{lemma:convex-up-bound}]
Let 
$$
\mu_{n,\beta}(A)= \mu_{\beta,n}^\omega(A)= \frac 1 {W_n} E_0\big[\exp\big(\sum_{i=1}^n \omega(i,S_i) - n \lambda(\beta)\big)\1_A\big]
$$
be the polymer (probability) measure and let
$(\theta_{n,y}\omega)(\cdot,\cdot)=\omega(n+\cdot, y + \cdot)$ denote the space-time shift. Then 
\begin{equation}\label{eq1}
\begin{aligned}
\E\bigg[f\bigg(\frac {W_{n+m}}{W_n}\bigg)\bigg | \mathcal F_n \bigg] 
&= \E\bigg[ f\bigg( \sum_{y \in \Z^d} \mu_{n,\beta}(S_n=y) W_m \circ \theta_{n,y}\bigg)\bigg| \mathcal F_n\bigg]
\end{aligned}
\end{equation}
By Jensen's inequality, 
$$
f\bigg( \sum_{y \in \Z^d} \mu_{n,\beta}(S_n=y) W_m \circ \theta_{n,y}\bigg) \leq  \sum_{y \in \Z^d} \mu_{n,\beta}(S_n=y) f\big(W_m \circ \theta_{n,y}\big)
$$
Therefore,
$$
\begin{aligned}
\E\bigg [ f \bigg(\frac {W_{n+m}}{W_n}\bigg)\bigg| \mathcal F_n \bigg] 
\leq \sum_{y \in \Z^d} \mu_{n,\beta}(S_n=y) \E\big[f\big(W_m \circ \theta_{n,y}\big)\big|\mathcal F_n\big] 
&=\E[f(W_m)].
\end{aligned}
$$
\end{proof}

 \begin{proof}[Proof of Lemma \ref{lemma:running-max}]
 	Let $\tau:=\inf\{n: W_n >t\}$, so that $W_\tau>t$ on the event $\{\tau < \infty\}$. Then 
$$
\begin{aligned}
\P[W_n > t \eps] &\geq \P\bigg[ \tau \leq n, \frac{W_n}{W_\tau} >\eps\bigg] \\
&= \sum_{k=1}^n \E\big[ \1_{\tau=k} \E\bigg(\1_{\frac{W_n}{W_k} >\eps} \bigg| \mathcal F_k\bigg)\bigg] \\
&\geq \sum_{k=1}^n \E\bigg[ \1_{\tau=k} \E\bigg[ f_{\eps}\big(\frac{W_n}{W_k}\big)\bigg| \mathcal F_k\bigg]\bigg], 
\end{aligned}
$$
where $f_{\eps}(x)= \big(\frac x \eps -1)\wedge 1$, which is a concave function that satisfies 
$\1_{[\eps,\infty)}(x) \geq f_{\eps}(x) \geq \1_{[2\eps,\infty)}(x) -  \1_{[0,\eps]}(x)$ whenever $x\geq 0$. 
 Now, using Lemma \ref{lemma:convex-up-bound}, we obtain 
\begin{equation}\label{eq2}
\P[W_n > t \eps] \geq \sum_{k=1}^n \E\bigg[ \1_{\tau=k} \E\big[ f_{\eps}(W_{n-k})\big] \geq \P[\tau \leq n] \inf_{k\in \N} \E[f_{\eps}(W_k)].
\end{equation} 
Moreover,
$$
\inf_{k\in \N} \E[f_{\eps}(W_k)] \geq  \E[ \inf_{k\in \N}   f_{\eps}(W_k)] \geq \P[ \inf_{k\in \N} W_k \geq 2 \eps]- \P[ \inf_{k\in \N} W_k \leq \eps ] \stackrel{\eps\to 0}\to \P[M_\infty >0] -  \P[M_\infty=0]=1,
$$
where  invoked the lower bound on $f_{\eps}(\cdot)$ and the fact that $\{W_\infty>0\}=\{\inf_{k\in \N} W_k>0\}$ is an event of probability 1. It follows that for some $\eps>0$, $\eta:= \inf_{k\in \N}\E[ f_{\eps}(W_k)]>0$. For such $\eps>0$, we conclude that for all $t>1$ and $n \in \N$,\begin{equation*}
	\P[M_n>t] =\P(\tau\leq n)\leq \frac 1 \eta\P[W_n > t\eps].
\end{equation*} Thus, for any $n\in \N$, 
$$
\E[M_n] \leq \frac 1 {\eps \eta} \E[W_n] +1 \leq \frac 1 {\eps \eta} +1. 
$$
Passing to $n\to\infty$, the left hand side converges to $\E[M_\infty]$. Thus, this quantity is also finite, as required.
 \end{proof}

\noindent{\bf Acknowledgement:} The research of the authors is funded by the Deutsche Forschungsgemeinschaft (DFG) under Germany's 
Excellence Strategy EXC 2044-390685587, Mathematics M\"unster: Dynamics-Geometry-Structure.


\begin{thebibliography}{WW98}
              
%
 



\bibitem{AN}
{\sc K.B. Athreya} and {\sc P.E. Ney.}
\newblock{Branching Processes.}
\newblock{Die Grundlehren der mathematischen Wissenschaften,} Band 196. Springer, New York (1972)

\bibitem{BC20}
{\sc E. Bates} and {\sc S. Chatterjee.}
\newblock{The endpoint distribution of directed polymers.} 
\newblock{\it Ann. Probab.} {\bf 48}(2), 817-871 (2020)


\bibitem{B89}
{\sc E. Bolthausen},
\newblock{A note on the diffusion of directed polymers in a random environment.}
\newblock{Commun. Math. Phys.} {\bf 123}, 529-534 (1989)


\bibitem{Biggins}
{\sc J. D. Biggins.}
\newblock{Growth rates in the branching random walk.}
\newblock{\it Z. Wahrsch. Verw. Gebiete.} {\bf 48}(1), 17-34 (1979)


\bibitem{Birkner}
{\sc M. Birkner.}
\newblock{A condition for weak disorder for directed polymers in random environment.}
\newblock{\it Electron.
Commun. Probab.} {\bf 9}, 22-25 (2004)




		
		
		
		
	\bibitem{BLM22}
		{\sc R. Bazaes}, {\sc I. Lammers} and {C. Mukherjee}.
		\newblock{Subcritical Gaussian Multiplicative Chaos in the Wiener Space: construction, moments and volume decay.}
		\newblock{\it Preprint}, arxiv: 2211.08996 (2022)


	
		
		

		
		
		
		\bibitem{C17}
		{\sc F. Comets.}
		\newblock{Directed polymers in random environments.}
		\newblock{St. Flour Lecture Notes}, \'Ecole d'\'Et\'e de Probabilit\'es de Saint-Flour XLVI -2016	
		
		
		
		
		
		
		
		\bibitem{CSY04}
  {\sc F. Comets}, {\sc T. Shiga} and {\sc N. Yoshida},
\newblock{Probabilistic analysis of directed polymers in a random environment: a review.}
\newblock{Stochastic analysis on large scale interacting systems},  
{Adv. Stud. Pure Math.} {\bf 39}, 115--142, 2004


		
		
		
		
		
		
		\bibitem{CY06}
		{\sc F. Comets} and {\sc N. Yoshida},\newblock{Directed polymers in random environment are diffusive in weak disorder.}
		\newblock{Ann. Probab.} {\bf 34}, 1746-1770, 2006.
		



\bibitem{IS88}
  {\sc J. Z. Imbrie} and {\sc T. Spencer},
\newblock{Diffusion of directed polymers in a random environment.}
\newblock{Journal of statistical Physics} {\bf 52}, nos. 3/4, 1988 





		
		\bibitem{J22}
		{\sc S. Junk}.
		\newblock{New Characterization of the Weak Disorder Phase of Directed Polymers in Bounded Random
		Environments}
		\newblock{Comm. Math. Phys.}{\bf 389}, no.2, 1087-1097, (2022)
		
		
		\bibitem{J22b}
		{\sc S. Junk}.
		\newblock{Fluctuations of partition functions of directed polymers in  weak disorder beyond the $L^2$ phase.}
		\newblock{Preprint}, arXiv: 2202.02907 (2022) 		
				
		
		
		
		
		
		
		
				
%
%
%
%
%
%
%
		\bibitem{MSZ16}
		{\sc C. Mukherjee}, {\sc A. Shamov} and {\sc O. Zeitouni.}
		{\em Weak and strong disorder for the stochastic heat equation and continuous directed polymers in $d\ge 3$}.
		\newblock{\it Electron. Commun. Probab.} 21:1-12 (2016) 
		
		
		
		
%
%
%
%
%
 \end{thebibliography}
\end{document}